\documentclass[11pt]{article}

\usepackage[letterpaper,margin=1in]{geometry}
\usepackage{amsmath,amssymb,amsthm,mathtools}
\usepackage{microtype}
\usepackage[colorlinks=true,linkcolor=blue,urlcolor=blue]{hyperref}

\newtheorem{theorem}{Theorem}[section]
\newtheorem{corollary}[theorem]{Corollary}
\newtheorem{lemma}[theorem]{Lemma}
\theoremstyle{definition}

\theoremstyle{remark}
\newtheorem{remark}[theorem]{Remark}
\newtheorem{definition}[theorem]{Definition}

\newcommand{\Var}{\operatorname{Var}}
\newcommand{\Ent}{\operatorname{Ent}}

\newcommand{\Z}{\mathbb Z}

\newcommand{\R}{\mathbb R}
\newcommand{\tmix}{T_{\rm mix}}

\usepackage{color}
\definecolor{rafid}{RGB}{0,102,204}      
\definecolor{antonio}{RGB}{102,0,204}

\title{Optimal mixing of the systematic scan dynamics via approximate tensorization of entropy}

\author{Antonio Blanca\thanks{Department of CSE, Pennsylvania State University, ablanca@cse.psu.edu.  Research supported in part by NSF CAREER grant CCF-2143762.} \and Md Tahmidur Rafid\thanks{Department of CSE, Pennsylvania State University, mxr5997@psu.edu. Research supported in part by NSF CAREER grant CCF-2143762} 
}

\date{}

\begin{document}

\maketitle

\thispagestyle{empty}

\begin{abstract}
We study the mixing time of the systematic scan dynamics for high-dimensional discrete distributions. This Markov chain updates coordinates sequentially according to a fixed predetermined order, in contrast to the Glauber dynamics that updates coordinates selected uniformly at random. The systematic scan is often favored in practice because it exhibits strong empirical performance, but its theoretical analysis remains far less developed than that of Glauber dynamics.

We take a step toward addressing this imbalance by showing that two standard functional notions of weak dependence between the coordinates of the distribution provide strong convergence guarantees for the systematic scan dynamics. First, we show that approximate tensorization of entropy implies optimal $O(\log n)$ mixing time for every scan order under standard marginal, connectivity, and bounded interaction degree assumptions about the distribution. Second, we show that approximate tensorization of variance yields a constant-factor contraction of the variance functional per scan, which in turn implies an optimal $O(1)$ relaxation time for the natural additive and multiplicative reversibilizations of the systematic scan dynamics. Compared with our entropy result, the variance bound improves the dependence on the maximum interaction degree from exponential to quadratic and requires weaker assumptions on the distribution.

As concrete applications of our results, we establish optimal $O(\log n)$ mixing of the systematic scan dynamics for bounded-degree antiferromagnetic two-spin systems in the tree-uniqueness region and for the ferromagnetic $q$-state Potts model on square boxes in $\mathbb{Z}^2$ throughout its subcritical regime. 
In both settings, slow mixing begins at the corresponding threshold, so the parameter regimes captured by entropy and variance tensorization are best possible.
These and many further applications follow from the fact that well-studied conditions such as spectral independence and strong spatial mixing are known to imply approximate tensorization of entropy and variance under standard assumptions.
\end{abstract}

\vfill

 \pagebreak

\setcounter{page}{1}

\section{Introduction}
Generating samples from high-dimensional distributions is a fundamental problem in computer science, statistics, and engineering. 
Let $V$ be a set of $n$ coordinates or variables, and for each $v \in V$, let $\Sigma_v$ denote the finite set of possible values for $v$. 
We consider distributions $\mu$ supported on
\begin{align*}
 \Omega\subseteq\prod\nolimits_{v\in V}\Sigma_v.
\end{align*}
A variety of algorithmic techniques have been developed for sampling from $\mu$, with \textit{Gibbs sampling} being one of the most prominent ones.
This approach iteratively updates a single variable using a \emph{heat-bath update}.
A heat-bath update replaces the value at a single coordinate with a sample from the conditional distribution under $\mu$ given the configuration in all the other coordinates. Formally, the single-site heat-bath operator at $v \in V$ is given by the $|\Omega| \times |\Omega|$ transition matrix:
\begin{align*}
 P_v(\sigma,\tau)
 =\mathbf 1\{\sigma_{V\setminus\{v\}}
                  =\tau_{V\setminus\{v\}}\}
   \mu(\tau_v\mid\sigma_{V\setminus\{v\}}),
\end{align*}
for all $\sigma,\tau \in \Omega$. We use $\sigma_A$ to denote the configuration of $\sigma$ on $A \subset V$ and $\sigma_v$ for $\sigma_{\{v\}}$.

A natural question is how the Gibbs sampler selects the update coordinate in each iteration. The classical \textit{Glauber dynamics} Markov chain selects the update location independently and uniformly at random in each iteration.
An appealing alternative is the \textit{systematic scan dynamics}, which updates coordinates sequentially following a fixed ordering. 
This Markov chain is often favored in practice because it performs well empirically, as it requires less randomness, exploits hardware locality, and is amenable to parallelization.
However, the theoretical understanding of the speed of convergence of the systematic scan dynamics is far less developed than that of the Glauber dynamics. 
In this work, we are interested in quantifying how fast the systematic scan dynamics converges to $\mu$. 

The ergodicity 
and convergence to $\mu$ 
of both the Glauber and systematic scan dynamics
is guaranteed by mild assumptions on the support of $\mu$.
In particular, it suffices for
$\Omega$ to be connected under feasible single-coordinate changes; that is, any two configurations in $\Omega$ can be joined by a sequence of configurations in $\Omega$ 
such that consecutive configurations differ in exactly one coordinate.  We call this property \emph{global connectedness} and assume it henceforth. 

Sampling from $\mu$ can be computationally hard, even approximately, so the systematic scan dynamics cannot be expected to converge rapidly in general. On the other hand, a large body of work shows that, under suitable weak-dependence conditions and sufficiently local interactions, the Glauber dynamics converges optimally usually in $\Theta(n\log n)$ steps. An important and largely open question is whether the same type of conditions also guarantee optimal mixing for the systematic scan dynamics.

The locality of the distribution $\mu$ can be encoded by the commutation relations among its heat-bath operators. We call a graph $G=(V,E)$ a \emph{commutation graph} for $\mu$ if, for every pair of distinct nonadjacent coordinates $u$ and $v$, we have
$P_uP_v=P_vP_u$.
Many well-studied high-dimensional distributions fit within this framework. For instance, for vertex-spin systems, which include the classical Ising, Potts, hard-core, and proper coloring models, the underlying interaction graph is precisely the commutation graph for the associated Gibbs distribution. The framework also includes Holant-type models whose variables are indexed by edges, including $b$-matchings, edge covers, and even-subgraph. We discuss the generality of our setting in more detail in Sections~\ref{subsec:spin} and~\ref{subsec:holant}.
 
On the other hand, the weak dependence among the coordinates of $\mu$ can be captured by a broad family of conditions, ranging from bounds on the norms of various influence matrices to sophisticated spatial mixing and functional conditions; see, e.g.,~\cite{dyer2008dobrushin,Hayes2006,matrixNorms,Cesi,combView2004,ALO,chenliu2023optimal}. In this work, we focus on two standard functional notions of weak dependence: \emph{approximate tensorization of variance} and \emph{approximate tensorization of entropy}. To formalize these notions, we introduce some notation first.   
For a function $f:\Omega\to\R$, let
\begin{align*}
 \mu[f] &:= \mathbb E_\mu[f]
 =\sum_{\sigma\in\Omega}\mu(\sigma)f(\sigma),~\text{and}\\
\Var_\mu(f)&:=\mu[(f-\mu[f])^2] = \mu[f^2] - \mu[f]^2
\end{align*}
denote its mean and variance with respect to $\mu$. For a coordinate $i \in V$, let $\Var_i(f)$ denote the conditional variance of $f$ when all coordinates except $i$ are fixed.
Note that $\Var_i(f)$ is a function of the configuration in $V\setminus\{i\}$.
Let $\mu[\Var_i(f)]$ denote the corresponding average variance.
With this notation in place, we can now formalize the notion of  approximate tensorization of variance.
\begin{definition}[Approximate tensorization of variance]
\label{ass:VAT}
We say that $\mu$ satisfies $C$-approximate tensorization of variance when for every function $f:\Omega\to\R$ we have
\begin{equation}\label{eq:VAT:intro}
 \Var_\mu(f)
 \leq
 C\sum_{i=1}^n\mu[\Var_i(f)].
\end{equation}
\end{definition}
Our first result states that approximate tensorization of variance yields a uniform contraction of the variance functional for the systematic  dynamics. 
For a fixed order $\pi=(v_1,\ldots,v_n)$ of the coordinates in $V$,
we relabel each $v_i$ as $i$ and define the transition matrix of the systematic scan dynamics as $P_\pi = P_{1}P_{2} \dots P_{n}$.

\begin{theorem}
\label{thm:product:intro}
Suppose that $\mu$ satisfies $C$-approximate tensorization of variance and admits a commutation graph of maximum degree $\Delta$. For every scan order $\pi$ and every real function $f$,
\begin{equation}\label{eq:var-contraction:intro}
 \Var_\mu(P_\pi f)
 \leq  \Big(1 - \frac{1}{1+C\Delta^2}\Big)\Var_\mu(f).
\end{equation}
\end{theorem}
For a reversible Markov chain, a variance contraction such as~\eqref{eq:var-contraction:intro} immediately yields a bound on the chain's \emph{relaxation time}. This quantity is the inverse of the \emph{spectral gap} of the transition matrix of the chain and governs its speed of convergence to stationarity from a ``warm start.''
The systematic scan dynamics is not necessarily reversible, so its additive and multiplicative reversibilization, $\frac{1}{2}(P_\pi + P_\pi^*)$ and $P_\pi^* P_\pi$, respectively, are often considered. Here
$P_\pi^* = P_{n}P_{{n-1}}\cdots P_{1}$ denotes the transition matrix
of the systematic scan dynamics with reversed scan order and corresponds to the adjoint of~$P_\pi$. We obtain  the following corollary.
\begin{corollary}
Under the assumptions of Theorem~\ref{thm:product:intro}, the additive and multiplicative reversibilizations of $P_\pi$ have relaxation time at most $2(1 + C\Delta^2)$. 
\end{corollary}
In particular, when $C$ and $\Delta$ are independent of the dimension $n$, this corollary yields an optimal $O(1)$ bound on the relaxation time of these chains.

Our second result provides tight \emph{mixing time} bounds for the systematic scan dynamics under a condition analogous to~\eqref{eq:VAT:intro} but for the entropy functional. 
The {mixing time} is a stronger measure of convergence to stationarity than the relaxation time, and captures
the speed of convergence to stationary from a \emph{worst-case} starting state. 
Formally, we let $P_\pi^t(\sigma,\cdot)$ denote the distribution of the Markov chain after $t$ steps starting from $\sigma \in \Omega$. The {\it mixing time} of the chain is defined as 
\[
\tmix(P_\pi) = \max\limits_{\sigma \in \Omega}\min \left\{ t \ge 0 : {\|P_\pi^t(\sigma,\cdot)-\mu\|}_{\textsc{tv}} \le 1/4 \right\},
\]
where $\|\cdot\|_{\textsc{tv}}$ denotes the total variation distance.
Turning our attention to entropy, for a function $f:\Omega\to\mathbb R_{\geq0}$, the entropy of $f$ with respect to $\mu$ is given by
\begin{align*}
\Ent_\mu( f) = \mu\left[f \cdot \log \left(\frac{f}{\mu[f]}\right)\right] = \mu[f \cdot \log f] - \mu[f] \cdot \log \mu[f];
\end{align*}
we use the conventions that $0\log0=0$ and $\Ent_\mu(0)=0$. As for variance, for $i \in V$, we let $\Ent_i(f)$ denote the conditional entropy of $f$ when all coordinates except $i$ are fixed and we let $\mu[\Ent_i(f)]$ denote the corresponding averaged entropy.

\begin{definition}[Approximate tensorization of entropy]
\label{ass:EAT}
We say that $\mu$ satisfies $C$-approximate tensorization of entropy when for every function $f:\Omega\to\R_{\ge 0}$ we have
\begin{equation*}
 \Ent_\mu(f)
 \leq
 C\sum_{i=1}^n\mu[\Ent_i(f)].
\end{equation*}
\end{definition}
We require two additional conditions on the distribution $\mu$. The first requires all single-coordinate conditional probabilities to be uniformly bounded away from zero, while the second ensures that the connectivity of the state space is preserved under conditioning. Both conditions are standard in prior work deriving algorithmic consequences from approximate tensorization of entropy; see, e.g.,~\cite{chenliu2023optimal,BCCPSV22}.

\begin{definition}[$b$-marginal boundedness]
\label{def:b-marginal}
Let $b\in(0,1)$. We say that $\mu$ is \emph{$b$-marginally bounded}
if, for every $U\subseteq V$, every configuration $\tau$ on $U$ with
$\mu(\sigma_U=\tau)>0$, every $i\notin U$, and every $a\in\Sigma_i$
that is feasible under this conditioning, we have
\begin{equation*}
 \mu(\sigma_i=a\mid\sigma_U=\tau)\geq b.
\end{equation*}
\end{definition}

\begin{definition}[Total connectedness]
\label{def:total-connectedness}
We say that $\mu$ is \emph{totally connected} if, for every nonempty set $S\subseteq V$ and every feasible configuration $\tau$ on $V \setminus S$, i.e, $\mu(\sigma_{V\setminus S}=\tau)>0$, the conditional support
\[
 \Omega_S^\tau
 :=
 \Big\{
   \eta\in\prod\nolimits_{v\in S}\Sigma_v:
   \mu\!\left(
     \sigma_S=\eta
     \,\middle|\,
     \sigma_{V\setminus S}=\tau
   \right)>0
 \Big\}
\]
is connected by feasible single-coordinate changes. 
\end{definition}

We are now ready to state our next result.
\begin{theorem}
\label{thm:entropy}
Suppose $\mu$ satisfies $C$-approximate tensorization of entropy, is totally connected, admits a commutation graph of maximum degree $\Delta$, and 
is $b$-marginally bounded.  For every scan order $\pi$, we have
\begin{equation*}
\tmix(P_\pi ) \leq 2(1 + CA)\big[ \log n +  \log\big(8 \log(1/b) \big)\big]
\end{equation*}
where $A = \Delta
\left(\frac{6\log(1/b)}{b^6}\right)^\Delta$.
\end{theorem}
In particular, if $C$ and $\Delta$ are bounded independently of $n$ and $b$ is uniformly bounded away from zero, then Theorem~\ref{thm:entropy} gives
$\tmix(P_\pi)=O(\log n)$. This logarithmic dependence on $n$ is optimal even for bounded-degree spin systems on bipartite graphs. Indeed, for the ferromagnetic Ising model on connected bounded-degree bipartite graphs, \cite[Lemma~4.1]{blanca2026rapid} establishes an $\Omega(\log n)$ lower bound for the multiplicative reversibilization of the \emph{alternating scan dynamics}, that corresponds to the systematic scan dynamics for an ordering $\pi$ that updates all vertices on one side of the bi-partition before updating those on the other side. The same lower bound transfers directly to the underlying nonreversible systemic scan dynamics.

Both Theorems~\ref{thm:entropy} and~\ref{thm:product:intro} also identify sharp parameter regimes for rapid mixing. For the hard-core and antiferromagnetic Ising models, approximate tensorization of entropy holds throughout the tree-uniqueness region, while approximate sampling is NP-hard outside of this regime~\cite{sly2012computational}; as such, the systematic scan dynamics provides an efficient sampler throughout the tractable side of this computational transition. 
In other settings where there is not underlying hardness, like the ferromagnetic Potts model on square boxes in $\mathbb Z^2$,
the analogous sharpness is dynamical. Specifically, approximate tensorization of entropy and variance hold throughout the subcritical regime, while systematic scan is exponentially slow outside of this regime~\cite{gheissari2018mixing,gaitonde2026comparison}. Thus, in both settings, these tensorization conditions reach the natural boundary for rapid mixing.

Before expanding on our proof ideas and applications of Theorems~\ref{thm:product:intro} and~\ref{thm:entropy}, we briefly place our results in the context of earlier work on systematic scan dynamics. One line of research compares the convergence rates of Glauber and systematic scan dynamics at varying levels of generality; see~\cite{amit1996convergence,gaitonde2026comparison,dyer2006systematic,diaconis2000analysis,bricmont1996statistical,chlebicka2025solidarity,guo2018layerwise}. These comparison results typically yield mixing time bounds for the systematic scan dynamics of order 
$O(n)$ steps or worse, and a step of the systematic scan dynamics corresponds to $n$ single-site heat-bath updates. 

A parallel line of work directly establishes sublinear, often polylogarithmic in $n$, mixing under conditions on the model parameters or the underlying graph structure~\cite{dyer2008dobrushin,Hayes2006,matrixNorms,blanca2026mixing,guo2018layerwise}, or under monotonicity assumptions~\cite{blanca2026rapid,BlancaCS18,peres2013can}. These results generally require substantially stronger hypotheses than approximate tensorization of entropy or variance and consequently apply only in narrower, often suboptimal, parameter regimes. For the Glauber dynamics, recent advances based on approximate tensorization of entropy have dramatically enlarged the parameter regimes in which optimal mixing is known, in some cases reaching the natural uniqueness/hardness threshold; see, e.g.,~\cite{chenliu2023optimal,feng2022rapid,chenfengyinzhang2024all}. Theorem~\ref{thm:entropy} does the same for the systematic scan dynamics, allowing us to convert approximate tensorization results into optimal $O(\log n)$ mixing guarantees and thereby substantially closing the parameter gap left by earlier direct analyses.

\subsection{Proof ideas}

We first explain the proof of the variance contraction in Theorem~\ref{thm:product:intro}, where the mechanism is most
transparent.  
First, $C$-approximate tensorization of variance 
gives
\begin{equation}\label{eq:atv:po}
 \Var_\mu(P_\pi f)
 \leq
 C\sum_i\mu[\Var_i(P_\pi f)].
\end{equation}
By the law of total variance,
\begin{equation*}
 \mu[\Var_i(g)] = \Var_\mu(g)-\Var_\mu(P_i g),
\end{equation*}
so $\mu[\Var_i(g)]$ is exactly the progress made by an update at $i$ when
progress is measured by variance contraction. Therefore, we see from~\eqref{eq:atv:po} that approximate tensorization
reduces the problem to measuring the progress 
that the $P_i$'s could still make
after one full scan.
We do this by bounding each term $\mu[\Var_i(P_\pi f)]$ in the right-hand side of ~\eqref{eq:atv:po} with a ``peeling argument''.

Suppose $i > 1$ and that $1$ is not a neighbor of $i$.  
Then $P_1$ and $P_i$ commute, so applying the two updates in either order gives the same result.  Viewed
in terms of variance contraction, applying $P_1$ first may remove some of
the variance that $P_i$ would otherwise remove, but it cannot create
additional variance for $P_i$ to remove.  
This is captured by the inequality 
\[
 \mu[\Var_i(P_1g)]
 \leq
 \mu[\Var_i(g)].
\]
Thus, $P_1$ can be ``peeled out'' from $\mu[\Var_i(P_\pi f)]$; i.e., $\mu[\Var_i(P_\pi f)] \le \mu[\Var_i(P_2\dots P_n f)]$. 
This is possible when $P_1$ and $P_i$
commute; for non-commuting Markov operators, a preliminary update can
create additional variance for $P_i$ to remove. 
(Simple low-dimensional examples indeed show that the inequality
$\mu[\Var_i(P_1g)]\leq\mu[\Var_i(g)]$ can fail without commutation.)
As such, if $1$ is adjacent to $i$, it cannot be removed for free, but it is possible to control its contribution via the conditional variance removed when $P_1$ is applied. 
We then continue to peel the product until $P_i$ is reached,
at which point the conditional variance at $i$ vanishes (i.e., $\mu[\Var_i(P_i \dots P_n f)] = 0$). In aggregate, we obtain an inequality of the form
\[
 \mu[\Var_i(P_\pi f)]
 \leq
 \Delta
 \sum_{{j<i,j\sim i}}
 \mu[\Var_j(P_{j+1:n}f)].
\]
Summing over $i$ incurs a second factor of $\Delta$, and the resulting
sum telescopes to
$\Var_\mu(f)-\Var_\mu(P_\pi f)$, which, together
with~\eqref{eq:atv:po}, gives the result.   

The entropy contraction proof uses the same insight and high-level structure. The mixing time bound in Theorem~\ref{thm:entropy} is obtained from an entropy contraction analog to~\eqref{eq:var-contraction:intro}.
This entropy contraction can be obtained by noting that
\[
 \mu[\Ent_i(g)]
 =\Ent_\mu(g)-\Ent_\mu(P_i g),
\]
and we can show again that commutation implies
\[
 \mu[\Ent_i(P_1g)]\leq\mu[\Ent_i(g)];
\]
for entropy though, this inequality is not immediate, but related estimates appeared in \cite{CaputoParisi,BCCPSV22}. 
For neighboring non-commuting updates, the resulting bound is weaker
than in the variance case, 
and it is controlled instead  through the conditional entropy on the
pair $\{i,j\}$. This is where the $b$-marginal boundedness
and total-connectedness assumptions come in to play.

\subsection{Applications to spin systems}
\label{subsec:spin}

We investigate next applications of our main results to spin systems. In particular, we consider
conditions under which approximate
tensorization of entropy and variance are known to hold in this context. Our two principal criteria are \emph{spectral independence} and \emph{strong spatial mixing}. We define these notions next.

General $q$-state spin systems form a general family of high-dimensional distributions. Fix an integer $q\geq 2$ and let $    [q]=\{1,\ldots,q\}$. Let $G=(V,E)$ be an undirected graph, and for every edge $\{x,y\}\in E$, let $A_{xy}\in\mathbb R_{\geq 0}^{q\times q}$ be a symmetric interaction matrix. For every $x\in V$, let $B_x\in\mathbb R_{\geq 0}^{q}$
be an external-field vector.
The Gibbs distribution $\mu$ on
$\Omega \subseteq [q]^V$ assigns to each configuration $\sigma\in\Omega$ probability
\[
    \mu(\sigma)=\frac{ \prod_{\{x,y\}\in E}
        A_{xy}(\sigma_x,\sigma_y)
    \prod_{x\in V} B_x(\sigma_x)}{Z_G},
\]
where $Z_G$ is the normalizing constant known as the partition function. Classical examples of spin systems include the Ising and Potts models, the hard-core model and proper $q$-colorings. 
The graph $G$ in the specification of the spin system corresponds to the commutation graph of the heat-bath operators.

\subsubsection{Spectral independence}
\label{sec:si-applications}

Spectral independence controls the cumulative influence among coordinates under pinnings. To formally define these notions, we require additional notation. Let 
$$
    \Omega_U
    =\big\{\tau\in \prod\nolimits_{x\in U} \Sigma_x:\mu(\sigma_U = \tau)>0\big\}.
$$
For $\tau\in\Omega_U$, let $\mu^\tau$ denote the conditional distribution on $V\setminus U$ given $\tau$, and set
\[
    X^\tau
    =\{(x,a) \in (V\setminus U)\times \Sigma_x: \mu^\tau(\sigma_x=a)>0\}.
\]

\begin{definition}[Influence matrix]
Let
$J^\tau\in\mathbb R^{X^\tau\times X^\tau}$ be the influence matrix given by
\[
J^\tau(x,a;y,a')
=
\begin{cases}
\mu^\tau(\sigma_y=a'\mid\sigma_x=a)
-\mu^\tau(\sigma_y=a') & x\neq y,\\
0 & x=y.
\end{cases}\]
\end{definition}
\begin{definition}[Spectral independence]
We say that the distribution $\mu$ is \emph{$\eta$-spectrally independent} if for every $U\subset V$ and every $\tau\in \Omega_U$, we have
\[
    \lambda_1(J^\tau)\leq\eta,
\]
where \(\lambda_1\) denotes the largest
eigenvalue.
More generally, given $\eta_0,\ldots,\eta_{n-2}\geq 0$, we say that
$\mu$ is \emph{$(\eta_0,\ldots,\eta_{n-2})$-spectrally independent} if
\[
\lambda_1(J^\tau)\leq \eta_{|U|}
\]
for every $U\subseteq V$ with $|U|\leq n-2$ and every feasible
$\tau\in\Omega_U$. For $c\geq0$ and $0\leq\eta<1$, we say that $\mu$ is
\emph{$(c,\eta)$-spectrally independent} if it is
$(\eta_0,\ldots,\eta_{n-2})$-spectrally independent with
\[
\eta_k=\min\{c,\eta(n-k-1)\}.\]

\end{definition}

Since its introduction in \cite{ALO}, spectral independence has found broad applications in the analysis of Markov chains for spin systems. This is largely because it can be shown to imply
the approximate tensorization of both variance and entropy.

\begin{theorem}[{\cite[Theorem 1.3]{ALO},~\cite[Theorem 1.1]{Jain}}]
\label{thm:si-to-avt}
Let $c\geq0$ and
$0\leq\eta<1$. Suppose $G=(V,E)$ has maximum degree $\Delta$ and that $\mu$ is $(c,\eta)$-spectrally independent.
Then, if $n \ge 50\lceil 2c\rceil\Delta$, $\mu$ satisfies 
$C$-approximate tensorization of variance with 
$$
    C=
    \frac{(25\Delta\lceil 2c\rceil)^{5\lceil 2c\rceil}}
         {\kappa(1-\eta)^{1+2c}},
$$
where $\kappa>0$ is a universal constant.
\end{theorem}

\begin{theorem}[{\cite[Theorem~2.9]{chenliu2023optimal},~\cite[Theorem~1.6]{BCCPSV22}}]
\label{thm:si-to-at}
Let $\eta>0$. Suppose $G=(V,E)$ has maximum degree $\Delta$ and that $\mu$ is $\eta$-spectrally independent,
$b$-marginally bounded and totally connected.
Then $\mu$ satisfies $C$-approximate tensorization of entropy with 
$
   C =\big(\frac{\Delta}{b}\big)^{O(1+\eta/b)}.
$
\end{theorem}

Theorems~\ref{thm:si-to-avt} and~\ref{thm:si-to-at},
when combined with our Theorems~\ref{thm:product:intro} and~\ref{thm:entropy},
set spectral
independence as a useful certificate for fast mixing of the
systematic scan dynamics. Recent works have established spectral independence with a variety of methods, including certain forms of correlation decay~\cite{ALO,ChenLiuVigodaContraction,chen2021rapid,feng2022rapid}, path coupling for local Markov chains~\cite{BCPSV,Liu}, and the stability of the partition function~\cite{chenliu2024stability}.
The list of bounded-degree spin systems for which spectral
independence is known is too extensive to catalog here.  We record instead four representative consequences of Theorem~\ref{thm:entropy}. Let $G=(V,E)$ be an $n$-vertex graph of maximum
degree $\Delta \ge 3$. Then, 

\begin{enumerate}
\item \emph{Hard-core model.}
The hard-core model with fugacity $\lambda>0$ is a distribution 
over the independent sets of a graph where each independent set $I$
 is assigned probability $\mu_{G,\lambda}(I)
    \propto \lambda^{|I|}$.
Let $\delta\in(0,1)$ and assume that $\lambda\leq(1-\delta)\lambda_c(\Delta)$ where $\lambda_c(\Delta)
    =\frac{(\Delta-1)^{\Delta-1}}{(\Delta-2)^\Delta}$ is the uniqueness threshold for $\Delta$-regular trees\footnote{In words, this means that on every infinite regular tree
of degree at most $\Delta$ there is a unique infinite-volume Gibbs
measure, or equivalently, distant boundary conditions have a vanishing
effect on the root marginal; see~\cite{weitzThesis} for a formal definition.}. 
Then, the systematic scan
dynamics for the hard-core model has mixing time
$O(\log n)$ \cite{ALO}.

\item \emph{Proper colorings.}
Let
$\mu_{G,q}$ be the uniform distribution over the proper $q$-colorings of the graph $G$.
When $q\geq(11/6-\varepsilon_0)\Delta$, where $\varepsilon_0 \approx 10^{-5}$,
the systematic scan dynamics for $\mu_{G,q}$ mixes in
$O(\log n)$ steps \cite{BCCPSV22}.

\item \emph{Antiferromagnetic two-spin systems.}
For $\sigma\in\{0,1\}^V$, let $m_s(\sigma)$ be the number of edges whose
endpoints both have spin $s\in\{0,1\}$, and let $n_1(\sigma)$ be the
number of vertices of spin $1$.  Given real numbers $\beta,\gamma,\lambda > 0$ such that $\beta\gamma<1$, an antiferromagnetic two-spin system is defined by the Gibbs distribution
$
    \mu_{G,\beta,\gamma,\lambda}(\sigma)
    \propto
    \beta^{m_1(\sigma)}\gamma^{m_0(\sigma)}
    \lambda^{n_1(\sigma)}.
$
If the system is in the tree uniqueness region
with a fixed gap\footnote{The fixed gap keeps the parameters
bounded away from the uniqueness threshold.}
$\delta>0$, 
then the systematic scan dynamics for
$\mu_{G,\beta,\gamma,\lambda}$ has mixing time $O(\log n)$
\cite{ChenLiuVigodaContraction}.

\item \emph{Ferromagnetic Potts model.}
For $q\geq2$ and $\sigma\in[q]^V$, 
the ferromagnetic $q$-state Potts model at inverse
temperature $\beta\geq0$ is given by the distribution
$
    \mu_{G,q,\beta}(\sigma) \propto e^{\beta M_G(\sigma)},$
    where 
$M_G(\sigma)$ denotes the number of monochromatic edges in $\sigma$. If
$
    \beta
    <\max\!\left\{
        \frac{2}{\Delta},
        \frac{1}{\Delta}\log\frac{q-1}{\Delta}
    \right\},
$
then the systematic scan dynamics for $\mu_{G,q,\beta}$ has mixing time $O(\log n)$ \cite{BCCPSV22}.
\end{enumerate}
It is worth noting that the parameter regimes in all four applications above match those in which Glauber dynamics is also known to mix optimally.
Beyond the bounded-degree setting, our
variance tensorization theorem (Theorem \ref{thm:product:intro}) has a quadratic dependence on the degree.
Therefore, for models like general antiferromagnetic two-spin
systems in the tree uniqueness region
with a fixed gap
$\delta>0$ that are known to be $\eta$-spectrally independent
with $\eta$ independent of $\Delta$ or $n$ (see, e.g.,~\cite[Theorem~2.10]{chenfengyinzhang2024all}),
Theorem \ref{thm:product:intro} applies and provides a $O(\Delta^2)$ bound for the relaxation time. 

\subsubsection{Strong spatial mixing on lattices}
\label{sec:ssm-applications}

There is a long line of work relating spatial mixing properties of
Gibbs distributions to functional inequalities, including approximate
tensorization of entropy and variance; see, e.g.,
\cite{StroockZegarlinski,MartinelliOlivieri,MartinelliOlivieriSchonmann,martinelli2004lectures,Cesi,CaputoParisi}.
Strong spatial mixing (SSM) is a standard formulation of
decay of correlations that, roughly speaking, asserts that the effect on the distribution of changing a spin in a distant region decays exponentially with the distance to where the change happens.
More precisely, 
consider the Gibbs distribution $\mu$ on a finite box
$\Lambda\subset\mathbb Z^d$. Let $B\subseteq\Lambda$, and let
$\mu_{\Lambda,B}^{\tau}$ denote the distribution on $B$ under 
the pinning $\tau$ on the exterior of $\Lambda$.  We say that $\mu$ satisfies SSM with constants
$A,c>0$ if whenever pinnings $\tau$ and $\tau'$ on the exterior of $\Lambda$ differ only at a boundary site $u$ of $\Lambda$, we have
\[
  \left\|
    \mu_{\Lambda,B}^{\tau}
    -\mu_{\Lambda,B}^{\tau'}
  \right\|_{\textsc{tv}}
  \leq A e^{-c\operatorname{dist}(u,B)},
\]
uniformly over the boxes $\Lambda$, $B$, $u$, and the pinnings.
It has been established that SSM implies approximate tensorization of variance and entropy on $\Z^d$~\cite{StroockZegarlinski,MartinelliOlivieri,MartinelliOlivieriSchonmann,Cesi,CaputoParisi}.

\begin{theorem}
\label{thm:ssm-to-at}
Suppose $\mu$ satisfies strong spatial mixing with constants
$A,c>0$. Then there exist finite constants
$C_{\textsc{V}}$ and $C_{\textsc{E}}$ 
such that $\mu_\Lambda^\tau$ satisfies $C_{\textsc{V}}$-approximate tensorization of variance and $C_{\textsc{E}}$-approximate tensorization of entropy for every finite box $\Lambda$ and pinning $\tau$.
The constants $C_{\textsc{V}}$ and $C_{\textsc{E}}$ depend on $d$, $A$, $c$, 
and the spin system specification but are independent of the dimensions of $\Lambda$ and the pinning $\tau$. 
\end{theorem}

In two dimensions, SSM is known to hold for the ferromagnetic Potts
model throughout the subcritical regime
$\beta<\beta_c(q)=\log(1+\sqrt q)$
for every integer $q\geq2$
\cite{BeffaraDuminilCopin,Alexander,MartinelliOlivieriSchonmann}.
The preceding theorem, combined with Theorem~\ref{thm:entropy}, therefore gives the following immediate
corollary.

\begin{corollary}
\label{cor:potts-mixing}
Fix $q\geq2$ integer and let $\beta<\beta_c(q)=\log(1+\sqrt q)$.
For every square
box $\Lambda\subset\mathbb Z^2$ with $n=|\Lambda|$, every pinning $\tau$ on the exterior of $\Lambda$, and every scan order $\pi$, the mixing time of the systematic scan
dynamics for $\mu_{\Lambda,\beta}^{\tau}$ satisfies
$\tmix(P_\pi) = O(\log n)$.
\end{corollary}
The dependence on $n$ in the upper bound in this corollary is sharp already
for $q=2$. Indeed, the $\Omega(\log n)$ lower bound in
\cite[Lemma~4.1]{blanca2026rapid} for the mixing time of the multiplicative
reversibilization of the scan order that updates all even vertices before
all odd vertices transfers directly to the corresponding nonreversible
systematic scan dynamics.

\subsection{Applications to edge-variable models}
\label{subsec:holant}

To illustrate the scope of our results beyond spin systems, we apply them to Holant problems.
These are graphical models defined over subsets of the edges
of a graph $H=(W,F)$.
In the binary symmetric case, each vertex $v$ is assigned a constraint function
$f_v:\mathbb N\to\mathbb R_{\ge 0}$ and each edge $e$ has a weight
$\lambda_e>0$. A subset $S \subseteq F$ is assigned probability 
$$
  \mu(S)
   \propto \prod\nolimits_{v\in W}f_v(|S\cap F_v|)
             \prod\nolimits_{e\in S}\lambda_e,
$$
where $F_v$ denotes the set of edges incident to $v$.
Classical examples of Holant problems 
include the monomer-dimer model, $b$-matchings, $b$-edge covers, and weighted even subgraphs. Taking $F$ as the coordinate set, the commutation graph for this class of distributions corresponds to the line graph $L(H)$; its maximum degree is at most $\Delta= 2(d-1)$ where $d$ is the maximum degree of $H$.

A fixed-parameter, totally connected, $b$-marginally bounded, and
$\eta$-spectrally independent Holant distribution with bounded
dependency degree is known to satisfy approximate tensorization of entropy with a
constant independent of $|F|$; this follows 
from a direct extension of~\cite[Theorem~2.9]{chenliu2023optimal} and has already been observed in~\cite{chenliu2024stability,chengu2024holant}.
Spectral independence in turn has been obtained for Holant models with log-concave constraint functions \cite[Theorem~13]{chengu2024holant}, and through zero-freeness and stability of the corresponding partition function
\cite[Theorems~1.7 and~3.2]{chenliu2024stability}.

Consequently, Theorem~\ref{thm:entropy} gives the following (non-exhaustive) applications to
fixed-parameter Holant families on bounded-degree graphs:
\begin{enumerate}
  \item \emph{Monomer--dimer model.}
  This is the Holant model with
  $f_v(k)=\mathbf 1\{k\leq1\}$.  Spectral independence is established
  in \cite[Theorem~2.10]{chenliu2023optimal}; hence, for every
  deterministic edge order, the systematic scan dynamics mixes in
  $O(\log |F|)$ time.

  \item \emph{Log-concave Holant models.}
  A Holant problem is said to be log-concave if
  for every $v \in W$ the constraint function $f_v$ satisfies 
  $f_v(k)^2\geq f_v(k-1)f_v(k+1)$, $f_v(0)>0$, and has no internal zeros (i.e.,  if $f_v(k_1) > 0$ and $f_v(k_2) > 0$ for some $k_1 < k_2$, $f_v(k) > 0$ for all
  integer $k \in \{k_1,\dots,k_2\}$). Log-concave models include weighted $b$-matchings and
  antiferromagnetic two-spin edge models.  
  Spectral independence is established in \cite[Theorem~13]{chengu2024holant} for these models, 
  and it thus follows that the mixing time of the systematic scan dynamics is $O(\log |F|)$.

  \item \emph{Subgraph world model.}
  This is the distribution where each subset $S \subseteq F$ is assigned probability $
    \mu(S)\propto\rho^{|\operatorname{odd}(S)|}\lambda^{|S|}$
       where $\rho\in(0,1]$ and $\lambda > 0$ are the model parameters.  The required spectral independence follows from the zero-freeness established in
  \cite{chenliu2024stability}.  Thus, the systematic scan dynamics for these models mixes in $O(\log |F|)$ time.
\end{enumerate}

\bigskip\noindent\textbf{AI Disclosure.} The authors used OpenAI Codex with GPT-5.6 Pro as a substantive,
human-directed research and writing aid during the development of this work.
The system was supplied with the authors' research notes 
with proposed approaches, 
and was used to explore proof strategies, refine candidate arguments, identify gaps and
counterexamples, and assist with drafting and revising the proofs. 
This assistance led to the main idea behind the variance contraction proof in Section~\ref{sec:var}, and the
entropy contraction proof in Section~\ref{sec:ent}. The authors independently verified all mathematical claims, calculations, and references in the paper and take full responsibility for its
accuracy, originality, and integrity.

\section{Preliminaries}
For the transition matrix of the systematic scan dynamics,
\begin{equation}\label{eq:physical}
 P_\pi=P_1P_2\cdots P_n,
\end{equation}
the site updates occur in the order $1,\ldots,n$, whereas when viewed as an operator, $P_\pi$ acts on $f$ from right to left in the order $n,\ldots,1$. For $i\leq j$, let
\begin{equation}\label{eq:consecutive-products}
 P_{i:j}:=P_iP_{i+1}\cdots P_j.
\end{equation}
For convenience, we set 
 $P_{i:j}:=I$ when $i>j$. Thus, $P_{1:n}=P_\pi$, $P_{i:n}=P_iP_{i+1:n}$, and $P_{n+1:n}f=f$.
Each $P_i$ is reversible with respect to $\mu$ and hence preserves $\mu$.
Consequently, the generally nonreversible product $P_\pi$ has stationary distribution $\mu$. The
assumed global connectedness of $\Omega$ 
under heat-bath moves makes $P_\pi$ irreducible, and since
$P_\pi(\sigma,\sigma)>0$ for every $\sigma\in\Omega$, we have that $P_\pi$ is also aperiodic. Hence, the systematic scan dynamics converges to $\mu$ from every initial state.

\subsection{Spectral gap}
Let $P$ be the transition matrix of an ergodic Markov chain with finite state space $\Omega$ and stationary distribution $\mu$.
We endow $\mathbb R^{|\Omega|}$ with the inner product
\[
 \langle f,g\rangle_\mu
 :=\sum_{x\in\Omega}f(x)g(x)\mu(x),
\]
and let $\langle f,f\rangle_\mu = \|f\|_{2,\mu}^2$. This gives the Hilbert space
$
 L^2(\mu)
 :=\bigl(\mathbb R^{|\Omega|},\langle\cdot,\cdot\rangle_\mu\bigr),
$
and $P$ defines a linear operator from $L^2(\mu)$ to itself.
Specifically, $P$ acts on 
functions $f:\Omega\to\mathbb R$ by
\begin{equation*}
 Pf(\sigma)
 :=\sum_{\tau\in\Omega}P(\sigma,\tau)f(\tau).
\end{equation*}
The adjoint of $P$ in $L^2(\mu)$ is the transition
matrix $P^*(x,y)
 =\frac{\mu(y)P(y,x)}{\mu(x)}$.
Consequently, $P$ is self-adjoint if and only if it is reversible with
respect to $\mu$.

When $P$ is reversible, we write its eigenvalues as
\[
 1=\lambda_1>\lambda_2\geq\cdots
 \geq\lambda_{|\Omega|}\geq-1.
\]
The strict inequality $\lambda_1>\lambda_2$ follows from irreducibility.
The \emph{absolute spectral gap} of $P$ is
\begin{equation}\label{eq:absolute-gap}
 \lambda(P):=1-\lambda_*,
 \qquad
 \lambda_*:=\max\bigl\{|\lambda_2|,
                         |\lambda_{|\Omega|}|\bigr\}.
\end{equation}
If $P$ is ergodic, meaning irreducible and aperiodic, then
$\lambda(P)>0$. The relaxation time is given by $\lambda(P)^{-1}$.

The operator $P$ is \emph{positive semidefinite} if $P=P^*$ and for every $f\in\mathbb R^{|\Omega|}$ we have $\langle f,Pf\rangle_\mu\geq0$.
In this case, all the eigenvalues of $P$ are nonnegative.
If $P$ is positive semidefinite, then its absolute spectral gap satisfies
\begin{equation}\label{eq:gap-variational}
 \lambda(P)
 =1-\lambda_2
 =\min_{\substack{f\in\mathbb R^{|\Omega|}\\
                  \Var_\mu(f)\neq0}}
   \frac{\mathcal E_P(f,f)}{\Var_\mu(f)},
\end{equation}
where,
for functions $f,h:\Omega\to\mathbb R$,
$\mathcal E_P(f,g) = \langle f,(I-P)g\rangle_\mu$ denotes 
the Dirichlet form associated to the pair $(P,\mu)$.

\subsection{Contraction of variance, entropy, and mixing times}

We say that $P$ contracts variance at rate $\delta \in (0,1)$ when
$$
\Var_\mu(P f) \le (1-\delta)\Var_\mu(f).
$$
It is a standard fact that such a contraction implies that the spectral gap of the additive and multiplicative reversibilizations of $P$ satisfy
\begin{align*}
 \lambda(\frac{P+P^*}{2})&\geq1-\sqrt{1-\delta},
 \\
 \lambda(P^*P)&\geq \delta; 
\end{align*}
see, e.g.,\cite[Proposition~2.3 and Remark~2.10]{Fill91}.
Similarly, $P$ contracts entropy at rate $\delta \in (0,1)$ when
\begin{equation}
\Ent_\mu(P f) \le (1-\delta)\Ent_\mu(f).
\end{equation}
If this holds for $P^*$ for every nonnegative function $f$, then
\begin{equation}
\label{eq:mixing}
\tmix(P)\le 1+\delta^{-1}[\log 8 + \log \log(1/\mu_*)],
\end{equation}
where $\mu_* = \min_{\sigma \in \Omega} \mu(\sigma)$; see~\cite[Lemma 2.4]{BCPSV}.

\section{Contraction of variance under the systematic scan dynamics}
\label{sec:var}

Let $Q_i:=I-P_i$. The following lemma collects the elementary properties about the operators $Q_i$ and $P_i$ used in our proof.

\begin{lemma}\label{lem:projection}
The following properties hold for each coordinate $i$:
\begin{enumerate}
    \item\label{lem:projection:0} $P_i^2=P_i=P_i^*$, $Q_i^2=Q_i=Q_i^*$, and $P_iQ_i=Q_iP_i=0$.
    \item $P_i$ and $Q_i$ are positive semidefinite. 
    \item\label{lem:projection:preserve} The operator $P_i$ preserves constants, nonnegativity, and the
    mean. 
    \item If $P_i$ and $P_j$ commute, then $Q_iP_j=P_jQ_i$, $P_iQ_j=Q_jP_i$,
        and $Q_iQ_j=Q_jQ_i$.
    \item\label{lem:projection:contract} $\|P_i f\|_{2,\mu} \le \|f\|_{2,\mu}$ and $\|Q_if\|_{2,\mu}\leq\|f\|_{2,\mu}$ for every function $f: \Omega \rightarrow \R$.
\end{enumerate}
\end{lemma}

We are now ready to establish that the systematic scan dynamics contracts the variance at a constant rate whenever $\mu$ satisfies $C$-approximate tensorization of variance.

\begin{theorem}
\label{thm:product}
Suppose that $\mu$ satisfies $C$-approximate tensorization of variance and admits a commutation graph of maximum degree at most $\Delta$.
For every scan order $\pi$, and every real function $f$,
\begin{equation*}
 \Var_\mu(P_\pi f)
 \leq  \Big(1 - \frac{1}{1+C\Delta^2}\Big)\Var_\mu(f).
\end{equation*}
\end{theorem}

\begin{proof}
Let
\begin{equation*}
 \mathcal E_i(f,f)
 :=\langle f,Q_if\rangle_\mu
 =\langle f,f\rangle_\mu-\langle f,P_if\rangle_\mu
\end{equation*}
denote the Dirichlet form for the update $P_i$. It is a standard fact
that $\mathcal E_i(f,f) =\mu[\Var_i(f)]$, so
$C$-approximate tensorization of variance is equivalent to
\begin{equation}
\label{eq:avt:alternative}
 \langle f-\mu[f],f-\mu[f]\rangle_\mu = \Var_\mu(f)
 \leq C\sum_{i=1}^n\mathcal E_i(f,f).
\end{equation}
Fix $\pi$ and let us assume first that $f$ has mean zero. 
By Lemma~\ref{lem:projection}(\ref{lem:projection:preserve}),
each $P_i$ preserves the mean, so $\mu[P_\pi f]=0$, and from~\eqref{eq:avt:alternative} we get
\begin{equation}\label{eq:variance-step-one}
 \langle P_\pi f,P_\pi f\rangle_\mu
 \leq C\sum_{i=1}^n\mathcal E_i(P_\pi f,P_\pi f).
\end{equation}
We bound each Dirichlet form $\mathcal E_i(P_\pi f,P_\pi f)$ next.
Suppose first that $i \not\sim j$.
Since by Lemma~\ref{lem:projection}, we have $\mathcal E_i(f,f)=\|Q_if\|_{2,\mu}^2$
and $Q_i P_j = P_j Q_i$, and also $P_{j:n}f = P_j(P_{j+1:n}f)$, we deduce that
\begin{align}\label{eq:variance-nonneighbor}
 \sqrt{\mathcal E_i(P_{j:n}f,P_{j:n}f)}  &=\|Q_iP_{j:n}f\|_{2,\mu}\notag\\
 &=\|P_jQ_iP_{j+1:n}f\|_{2,\mu}\notag\\
 &\leq\|Q_iP_{j+1:n}f\|_{2,\mu}
  =\sqrt{\mathcal E_i(P_{j+1:n}f,P_{j+1:n}f)},
 \end{align}
 where the inequality follows from Lemma~\ref{lem:projection}(\ref{lem:projection:contract}).
On the other hand, if $j\sim i$, from the triangle inequality and the contraction of $Q_i$ we obtain
\begin{align}
 \sqrt{\mathcal E_i(P_{j:n}f,P_{j:n}f)}
 &=\|Q_iP_{j:n}f\|_{2,\mu}\notag\\
 &=\|Q_iP_{j+1:n}f-Q_iQ_jP_{j+1:n}f\|_{2,\mu}\notag\\
 &\leq\|Q_iP_{j+1:n}f\|_{2,\mu}
       +\|Q_iQ_jP_{j+1:n}f\|_{2,\mu}\notag\\
 &\leq\sqrt{\mathcal E_i(P_{j+1:n}f,P_{j+1:n}f)}
       +\sqrt{\mathcal E_j(P_{j+1:n}f,P_{j+1:n}f)}.
 \label{eq:variance-neighbor}
\end{align}
This says that a neighboring update increases the square root of the
$i$-th Dirichlet form by at most the square root of its own Dirichlet form.
Let $j_1>\cdots>j_r$ list the neighbors of $i$ that appear earlier in
$\pi$; then, each $j_\ell<i$ and $r\leq\Delta$. Iterating
\eqref{eq:variance-nonneighbor}-\eqref{eq:variance-neighbor} until reaching $P_i$, we obtain
\begin{align*}
\sqrt{
 \mathcal E_i(P_{1:n}f,P_{1:n}f)
}
&\leq
\sqrt{
 \mathcal E_i(P_{i:n}f,P_{i:n}f)
}
+\sum_{\ell=1}^r
 \sqrt{
  \mathcal E_{j_\ell}
  (P_{j_\ell+1:n}f,P_{j_\ell+1:n}f)
 }.
\end{align*}
Now, since $Q_iP_i=0$ by Lemma~\ref{lem:projection}(\ref{lem:projection:0}), we get
\begin{align*}
\mathcal E_i(P_{i:n}f,P_{i:n}f)
=
\bigl\|Q_iP_{i:n}f\bigr\|_{2,\mu}^2
=
\bigl\|Q_iP_iP_{i+1:n}f\bigr\|_{2,\mu}^2
=0.
\end{align*}
Therefore,
\begin{equation*}
\sqrt{
 \mathcal E_i(P_\pi f,P_\pi f)
}
\leq
\sum_{\ell=1}^r
 \sqrt{
  \mathcal E_{j_\ell}
  (P_{j_\ell+1:n}f,P_{j_\ell+1:n}f)
 },
\end{equation*}
and the Cauchy--Schwarz inequality gives
\begin{align}
\mathcal E_i(P_\pi f,P_\pi f)
&\leq
r\sum_{\ell=1}^r
 \mathcal E_{j_\ell}
 (P_{j_\ell+1:n}f,P_{j_\ell+1:n}f)
\leq
\Delta
\sum_{{j<i,j\sim i}}
\mathcal E_j(P_{j+1:n}f,P_{j+1:n}f).
\label{eq:local-charge}
\end{align}
Summing \eqref{eq:local-charge} over $i$ yields
\begin{align}
\sum_{i=1}^n
 \mathcal E_i(P_\pi f,P_\pi f)
&\leq
\Delta
\sum_{i=1}^n
\sum_{{j<i,j\sim i}}
 \mathcal E_j(P_{j+1:n}f,P_{j+1:n}f)
\leq
\Delta^2
\sum_{j=1}^n
 \mathcal E_j(P_{j+1:n}f,P_{j+1:n}f),
\label{eq:global-charge}
\end{align}
since for each fixed $j$, the term
$\mathcal E_j(P_{j+1:n}f,P_{j+1:n}f)$
appears once for every neighbor $i>j$, and there are at most $\Delta$
such neighbors. Note that the index $j=n$ does not occur in the double sum, since
there is no $i>n$, but we use the convention $P_{n+1:n}=I$ to add the nonnegative term 
$$\mathcal E_n(P_{n+1:n}f,P_{n+1:n}f)=\mathcal E_n(f,f)$$
to the final sum. 

The sum on the right-hand side of~\eqref{eq:global-charge} telescopes since
\begin{align*}
 \mathcal E_j(P_{j+1:n}f,P_{j+1:n}f)
 &=\langle P_{j+1:n}f,P_{j+1:n}f\rangle_\mu
  -\langle P_{j+1:n}f,P_jP_{j+1:n}f\rangle_\mu \\
  &= \langle P_{j+1:n}f,P_{j+1:n}f\rangle_\mu
  -\langle P_{j:n}f,P_{j:n}f\rangle_\mu,
\end{align*}
and therefore
\begin{equation}\label{eq:dissipation}
 \sum_{j=1}^n\mathcal E_j(P_{j+1:n}f,P_{j+1:n}f)
 =\langle f,f\rangle_\mu-\langle P_\pi f,P_\pi f\rangle_\mu.
\end{equation}
Combining \eqref{eq:variance-step-one}, \eqref{eq:global-charge}, and
\eqref{eq:dissipation} we obtain
\[
 \langle P_\pi f,P_\pi f\rangle_\mu
 \leq C\Delta^2
 \bigl(\langle f,f\rangle_\mu-\langle P_\pi f,P_\pi f\rangle_\mu\bigr).
\]
Rearranging proves that
\begin{equation}\label{eq:one-sweep-quadratic}
 \Var_\mu ( P_\pi f) = \langle P_\pi f,P_\pi f\rangle_\mu
 \leq \Big(1 - \frac{1}{1+C\Delta^2}\Big) \langle f,f\rangle_\mu = \Big(1 - \frac{1}{1+C\Delta^2}\Big) \Var_\mu(f).
\end{equation}

For an arbitrary $f:\Omega\to\mathbb R$, since $P_\pi$ preserves constants,
$P_\pi (f-\mu[f])=P_\pi f-\mu[f]$, so applying
\eqref{eq:one-sweep-quadratic} to $f-\mu[f]$ gives the result for arbitrary functions $f:\Omega\to\mathbb R$.
\end{proof}

\section{Contraction of entropy under the systematic scan dynamics}
\label{sec:ent}

We start by recording several definitions and facts related to the entropy functional that are needed in our proof. 
For a subset of vertices $A\subseteq V$, let
$\mathcal F_A$ denote the
$\sigma$-algebra generated by the spins outside of $A$.
The conditional entropy is defined as
\begin{align}
 \Ent_A(f)
 &:=\mu[f\log f\mid\mathcal F_A]
   -\mu[f\mid\mathcal F_A]\log\mu[f\mid\mathcal F_A]
 =\mu\!\Big[
     f\log\frac{f}{\mu[f\mid\mathcal F_A]}
     \,\Big|\,\mathcal F_A
   \Big].\label{eq:block-entropy}
\end{align}
This is a function of the spins outside $A$; indeed, after fixing those spins, $\Ent_A(f)$ is the entropy of the function $f$ under the conditional distribution on $A$. We use $\mu[\Ent_A(f)]$ for the
averaged entropy remaining when the spins outside $A$ are fixed.

\begin{lemma}\label{lem:entropy-facts}
For every nonnegative function $f$, the following properties hold.
\begin{enumerate}

\item\label{item:entropy-nonnegative}
$\Ent_\mu(f)\geq0$.

\item\label{item:entropy-homogeneous}
For every $c\geq0$, $\Ent_\mu(cf)=c\Ent_\mu(f)$.


\item\label{item:entropy-chain-rule}
For every $i\in A\subseteq V$, we have
\[
 \Ent_A(f)
 =\Ent_A\!\bigl(\mu[f\mid\mathcal F_i]\bigr)
  +\mu[\Ent_i(f)\mid\mathcal F_A].
\]

\item\label{item:block-domination}
For every $i\in A\subseteq V$,
$\mu[\Ent_i(f)]\leq\mu[\Ent_A(f)]$.
\end{enumerate}
\end{lemma}

Parts~\ref{item:entropy-nonnegative} and~\ref{item:entropy-homogeneous} of this lemma are standard facts about the entropy functional. Part~\ref{item:entropy-chain-rule} is the entropy decomposition
in \cite[Lemma~3.1]{CaputoParisi}, applied under the conditional law on $A$; 
see also \cite[Lemma~2.3]{BCCPSV22}.
Part~\ref{item:block-domination} is a special case of
\cite[Lemma~2.4]{BCCPSV22}.
 We also use the following two-site specialization of
\cite[Lemma~4.2]{chenliu2023optimal}. 
\begin{lemma}
\label{lem:two-site}
Suppose that $\mu$ is $b$-marginally bounded and totally connected.
Then, for every edge $\{i,j\}$ and every $f\geq0$,
\begin{equation*}
 \Ent_{\{i,j\}}(f)
 \leq
 \frac{6\log(1/b)}{b^6}\,
 \mu\!\left[
   \Ent_i(f)+\Ent_j(f)
   \,\middle|\,
   \mathcal F_{\{i,j\}}
 \right].
\end{equation*}
\end{lemma}
Next, we record a monotonicity property of the entropy functional under
commuting heat-bath updates. Related estimates appear in
\cite[Eq.~(3.6)]{CaputoParisi} and
\cite[Lemma~2.7]{BCCPSV22}. In those settings, the
required commutation relation follows from an underlying product
structure, and the inequalities are proved using the variational
characterization of entropy instead. Since we assume commutation directly, we include a self-contained proof based on the log-sum inequality in Section~\ref{subsec:mono}.

\begin{lemma}
\label{lem:commuting-entropy}
If $P_iP_j=P_jP_i$, then for every $f\geq0$, we have
\begin{equation*}
 \mu[\Ent_i(P_jf)]\leq\mu[\Ent_i(f)].
\end{equation*}
\end{lemma}

We are now ready to show that the systematic scan dynamics contracts the entropy functional when approximate tensorization of entropy holds. 

\begin{theorem}
\label{thm:ordered}
Suppose $\mu$ satisfies $C$-approximate tensorization of entropy, and that
it is $b$-marginally bounded, totally connected, and admits a commutation graph of maximum degree at most $\Delta$. For every scan order $\pi$ and every
function $f\geq0$
\begin{equation}\label{eq:ordered-contraction}
\Ent_\mu(P_\pi f)
\leq
\left(
1-\frac{1}{
1+C A
}
\right)
\Ent_\mu(f),
\end{equation}
where $A = \Delta
\left(\frac{6\log(1/b)}{b^6}\right)^\Delta$.
\end{theorem}

\begin{proof}
Fix $\pi$ and $f\geq0$. 
From the assumption that $\mu$ satisfies $C$-approximate tensorization of entropy, we get
\begin{equation}\label{eq:entropy-step-one}
 \Ent_\mu(P_\pi f)
 \leq C\sum_{i=1}^n\mu[\Ent_i(P_\pi f)].
\end{equation}
If $j\not\sim i$, by Lemma~\ref{lem:commuting-entropy} we have
\begin{equation}\label{eq:entropy-nonneighbor}
 \mu[\Ent_i(P_{j:n}f)]
 \leq
 \mu[\Ent_i(P_{j+1:n}f)].
\end{equation}
If $j\sim i$, then
\begin{align}
 \mu[\Ent_i(P_{j:n}f)]
 &\leq
 \mu[\Ent_{\{i,j\}}(P_{j:n}f)]
 &&
 \text{Lemma~\ref{lem:entropy-facts}
       (\ref{item:block-domination})}
 \notag\\
 &=
 \mu[\Ent_{\{i,j\}}(P_{j+1:n}f)]
 -
 \mu[\Ent_j(P_{j+1:n}f)]
 &&
 \text{Lemma~\ref{lem:entropy-facts}
       (\ref{item:entropy-chain-rule})}
 \notag\\
 &\leq
 \mu[\Ent_{\{i,j\}}(P_{j+1:n}f)]
 &&
 \text{Lemma~\ref{lem:entropy-facts}
       (\ref{item:entropy-nonnegative})}
 \notag\\
 &\leq
 \frac{6\log(1/b)}{b^6}
 \bigl(
   \mu[\Ent_i(P_{j+1:n}f)]
   +
   \mu[\Ent_j(P_{j+1:n}f)]
 \bigr)
 &&
 \text{Lemma~\ref{lem:two-site}}.
 \label{eq:neighbor-recurrence}
\end{align}

Let $j_1>\cdots>j_r$ list the neighbors of vertex $i$ that appear earlier in $\pi$; then, $j_\ell<i$, and $r\leq\Delta$. 
Iterating \eqref{eq:entropy-nonneighbor}-\eqref{eq:neighbor-recurrence} until reaching $P_i$, we obtain
\begin{align}
 \mu[\Ent_i(P_{1:n}f)]
 &\leq
 \left(\frac{6\log(1/b)}{b^6}\right)^r
 \mu[\Ent_i(P_{i:n}f)]
 +
 \sum_{\ell=1}^r
 \left(\frac{6\log(1/b)}{b^6}\right)^{r-\ell+1}
 \mu[\Ent_{j_\ell}(P_{j_\ell+1:n}f)].
 \label{eq:entropy-iteration}
\end{align}

By Lemma~\ref{lem:entropy-facts}(\ref{item:entropy-chain-rule}), applied
to $P_{i:n}f$, we have
\begin{align*}
 \mu[\Ent_i(P_{i:n}f)]
 &=
 \Ent_\mu(P_{i:n}f)
 -
 \Ent_\mu\!\left(
   \mu[P_{i:n}f\mid\mathcal F_i]
 \right)\\
 &=
 \Ent_\mu(P_{i:n}f)
 -
 \Ent_\mu(P_iP_{i:n}f)\\
 &=
 \Ent_\mu(P_{i:n}f)-\Ent_\mu(P_{i:n}f)
 =
 0.
\end{align*}
Thus the first term on the right-hand side of
\eqref{eq:entropy-iteration} vanishes and we obtain
\[
 \mu[\Ent_i(P_{1:n}f)]
 \leq
 \sum_{\ell=1}^r
 \left(\frac{6\log(1/b)}{b^6}\right)^{r-\ell+1}
 \mu[\Ent_{j_\ell}(P_{j_\ell+1:n}f)].
\]
We shall assume without loss of generality that $b\leq1/2$. Then, $\frac{6\log(1/b)}{b^6}\geq1$.
Moreover, $r\leq\Delta$, and the local entropic terms in the summation are nonnegative by
Lemma~\ref{lem:entropy-facts}(\ref{item:entropy-nonnegative}).
Hence,
\begin{equation}\label{eq:entropy-local-charge}
 \mu[\Ent_i(P_\pi f)]
 \leq
 \left(\frac{6\log(1/b)}{b^6}\right)^\Delta
 \sum_{{j<i,j\sim i}}
 \mu[\Ent_j(P_{j+1:n}f)].
\end{equation}
Summing over $i$ gives
\begin{align}
 \sum_{i=1}^n\mu[\Ent_i(P_\pi f)]
 &\leq
 \left(\frac{6\log(1/b)}{b^6}\right)^\Delta
 \sum_{i=1}^n
 \sum_{{j<i,j\sim i}}
 \mu[\Ent_j(P_{j+1:n}f)]\notag\\
 &\leq
 \Delta
 \left(\frac{6\log(1/b)}{b^6}\right)^\Delta
 \sum_{j=1}^{n}
 \mu[\Ent_j(P_{j+1:n}f)],\label{eq:main:ineq}
\end{align}
because for each fixed $j$, the term
$\mu[\Ent_j(P_{j+1:n}f)]$ appears once for every neighbor $i>j$.
We remark that the index $j=n$ does not appear in the double sum because there is no index $i>n$. Nevertheless, with the convention that $P_{n+1:n}=I$, we have $\mu[\Ent_n(P_{n+1:n}f)]=\mu[\Ent_n(f)] \ge 0$, so this term can be added to the right-hand
side; this is useful because the sum on the right now
contains the entropy loss from every update in the scan.

By Lemma~\ref{lem:entropy-facts}(\ref{item:entropy-chain-rule}),
\begin{align}
 \mu[\Ent_i(P_{i+1:n}f)]
 &=
 \Ent_\mu(P_{i+1:n}f)
 -
 \Ent_\mu\!\left(
   \mu[P_{i+1:n}f\mid\mathcal F_i]
 \right)
 \notag\\
 &=
 \Ent_\mu(P_{i+1:n}f)
 -
 \Ent_\mu(P_{i:n}f),
 \label{eq:one-update-entropy-loss}
\end{align}
since $
 \mu[P_{i+1:n}f\mid\mathcal F_i]
 =
 P_iP_{i+1:n}f
 =
 P_{i:n}f.$
Then, summing over $i=1,\dots,n$ and telescoping we obtain
\begin{equation}\label{eq:entropy-loss}
 \sum_{i=1}^{n}\mu[\Ent_i(P_{i+1:n}f)]
 =\Ent_\mu(f)-\Ent_\mu(P_\pi f).
\end{equation}
Combining \eqref{eq:entropy-step-one},
\eqref{eq:main:ineq}, and \eqref{eq:entropy-loss} gives
\[
 \Ent_\mu(P_\pi f)
 \leq  C\Delta
 \left(\frac{6\log(1/b)}{b^6}\right)^\Delta\bigl(\Ent_\mu(f)-\Ent_\mu(P_\pi f)\bigr),
\]
and rearranging proves \eqref{eq:ordered-contraction}.
\end{proof}

We conclude this section with the proof of Theorem~\ref{thm:entropy}.

\begin{proof}[Proof of Theorem~\ref{thm:entropy}]
    Follows from Theorem~\ref{thm:ordered} and~\eqref{eq:mixing}.
\end{proof}

\subsection{Monotonicity of entropy under commuting updates: proof of Lemma~\ref{lem:commuting-entropy}}
\label{subsec:mono}

We provide in this section the proof Lemma~\ref{lem:commuting-entropy} thus completing the proof of Theorem~\ref{thm:ordered}.

\begin{proof}[Proof of Lemma~\ref{lem:commuting-entropy}]
First suppose that $f>0$, so that all conditional means appearing
below are positive.
Taking expectation in~\eqref{eq:block-entropy}, we obtain
\begin{equation}\label{eq:entropy-loss-log}
 \mu[\Ent_i(f)]
 =
 \mu\!\left[
   f\log\frac{f}{\mu[f\mid\mathcal F_i]}
 \right].
\end{equation}

For positive numbers
$a_1,\ldots,a_r$ and $b_1,\ldots,b_r$, and weights
$\theta_k\geq0$ satisfying $\sum_k\theta_k=1$, the log-sum inequality (see, e.g.,~\cite[Theorem~2.7.1]{CT06}) yields that
\begin{equation}\label{eq:log-sum}
 \left(\sum_k\theta_k a_k\right)
 \log\frac{\sum_k\theta_k a_k}{\sum_k\theta_k b_k}
 \leq
 \sum_k\theta_k a_k\log\frac{a_k}{b_k}.
\end{equation}

Fix the spins outside $j$ and apply \eqref{eq:log-sum} with
the conditional probabilities of the spin at $j$ as weights.
Take $a_k$ to be the corresponding values of $f$, and $b_k$
to be the corresponding values of $\mu[f\mid\mathcal F_i]$.
The two weighted averages on the left are then
$\mu[f\mid\mathcal F_j]$ and
$\mu[\mu[f\mid\mathcal F_i]\mid\mathcal F_j]$.
Consequently,
\begin{equation}
\label{log:ineq}
 \mu[f\mid\mathcal F_j]\,
 \log
 \frac{\mu[f\mid\mathcal F_j]}
      {\mu[\mu[f\mid\mathcal F_i]\mid\mathcal F_j]}
 \leq
 \mu\!\left[
   f\log\frac{f}{\mu[f\mid\mathcal F_i]}
   \,\middle|\,\mathcal F_j
 \right].
\end{equation}

Since $P_i$ and $P_j$ are
the conditional-expectation operators given $\mathcal F_i$
and $\mathcal F_j$ (i.e., $P_if=\mu[f\mid\mathcal F_i]$ and $P_jf=\mu[f\mid\mathcal F_j]$), the assumption that $P_iP_j=P_jP_i$ implies that
\[
 \mu[\mu[f\mid\mathcal F_i]\mid\mathcal F_j]
 =
 \mu[\mu[f\mid\mathcal F_j]\mid\mathcal F_i].
\]
It then follows from~\eqref{eq:entropy-loss-log} that
\begin{align*}
 \mu[\Ent_i(P_jf)]
 = \mu\!\left[
   \mu[f\mid\mathcal F_j]\,
   \log
   \frac{\mu[f\mid\mathcal F_j]}
        {\mu[\mu[f\mid\mathcal F_j]\mid\mathcal F_i]}
 \right]= 
 \mu\!\left[
   \mu[f\mid\mathcal F_j]\,
   \log
   \frac{\mu[f\mid\mathcal F_j]}
        {\mu[\mu[f\mid\mathcal F_i]\mid\mathcal F_j]}
 \right],
\end{align*}
and that 
\begin{align*}
 \mu\!\left[
   \mu\!\left[
     f\log\frac{f}{\mu[f\mid\mathcal F_i]}
     \,\middle|\,\mathcal F_j
   \right]
 \right]
 &=
 \mu\!\left[
   f\log\frac{f}{\mu[f\mid\mathcal F_i]}
 \right]
 =
 \mu[\Ent_i(f)].
\end{align*}
Combining these two last identities, and averaging~\eqref{log:ineq}
yields that
$$
 \mu[\Ent_i(P_jf)] \le \mu[\Ent_i(f)],
$$
as claimed. 

Finally, let us consider the case when $f\geq0$. Apply the result just proved to
$f+\varepsilon$, where $\varepsilon>0$.
Conditional expectation preserves constants, so
$P_j(f+\varepsilon)=P_jf+\varepsilon$, and therefore
\[
 \mu[\Ent_i(P_jf+\varepsilon)]
 \leq
 \mu[\Ent_i(f+\varepsilon)].
\]
Since $x\log x$ extends continuously to zero with value zero,
the definition of conditional entropy shows that both sides
are continuous in $\varepsilon$ on the finite state space.
Letting $\varepsilon\downarrow0$ proves the claim.
\end{proof}

\begin{remark}
We remark that the lemma can also be proved using the data processing inequality for the relative entropy. Indeed, we can apply this inequality to $f$ and $\mu[f\mid\mathcal F_i]$ under the reversible operator $P_j$, and then use commutation to obtain
$P_j\mu[f\mid\mathcal F_i]=\mu[P_jf\mid\mathcal F_i]$; see \cite[Theorem~7.4]{PolyanskiyWu} for this version of the data processing inequality.
\end{remark}

\bibliographystyle{alpha}
\bibliography{references}
\end{document}